\documentclass[12pt]{article}
\usepackage{latexsym,amsmath,amssymb}

\newif\ifdviwin

\usepackage[latin1]{inputenc}
\usepackage[english]{babel}
\usepackage{indentfirst}
\usepackage[mathscr]{eucal}
\usepackage{amssymb,amsmath,amsfonts}
\usepackage{fancybox,fancyhdr}
\usepackage{graphicx}

\newif\ifdviwin

\dviwintrue

\def\cA{\mathcal{A}}
\def\cU{\mathcal{U}}

\def\cS{\mathcal{S}}

\def\cL{\mathcal{L}}

\def\cN{\mathcal{N}}

\let\8=\infty \let\0=\emptyset

\let\landa=\lambda
\let\alfa=\alpha

\let\parc=\partial

\def\ep{\varepsilon}

\def\landa{\lambda}

\def\lap{\Delta}
\def\flecha{\rightarrow}
\def\esiz{\langle}
\def\esde{\rangle}

\def\cte.{\mathop{\rm cte.}\nolimits}

\def\N{\mathbb{N}}
\def\L{\mathbb{L}}

\def\R{\mathbb{R}}
\def\Z{\mathbb{Z}}
\def\C{\mathbb{C}}
\def\A{\mathbb{A}}
\def\D{\mathbb{D}}

\def\H{\mathbb{H}}
\def\S{\mathbb{S}}

 \newtheorem{defi}{Definition}
 \newtheorem{teo}[defi]{Theorem}
 \newtheorem{pro}[defi]{Proposition}
 \newtheorem{cor}[defi]{Corollary}
 \newtheorem{lem}[defi]{Lemma}
 
 \newtheorem{remark}[defi]{Remark}

 \newenvironment{proof}{\rm \trivlist \item[\hskip \labelsep{\it
      Proof}:]}{\par\nopagebreak \hfill $\Box$ \endtrivlist}

\numberwithin{equation}{section}

\begin{document}
\mbox{}\vspace{0.4cm}\mbox{}

\begin{center}
\rule{15cm}{1.5pt}\vspace{0.5cm}

\renewcommand{\thefootnote}{\,}
{\Large \bf Surfaces of constant curvature in $\R^3$ \\[0.3cm] with isolated singularities }\\

\vspace{0.5cm} {\large José A. Gálvez$^a$, Laurent Hauswirth$^b$ and Pablo Mira$^c$}\\
\vspace{0.3cm} \rule{15cm}{1.5pt}
\end{center}
  \vspace{0.5cm}
{\small $\mbox{}^a$ Departamento de Geometría y Topología,
Universidad de Granada, E-18071 Granada, Spain. \\ e-mail:
jagalvez@ugr.es} \vspace{0.2cm}

{\small \noindent $\mbox{}^b$  Département de Mathématiques
Université de Marne-la-Vallée Cité Descartes - 5 boulevard Descartes
Champs-sur-Marne 77454 Marne-la-Vallée Cedex 2 France \\ e-mail:
laurent.hauswirth@univ-mlv.fr} \vspace{0.2cm}

{\small \noindent $\mbox{}^c$ Departamento de Matemática Aplicada y
Estadística,
Universidad Politécnica de Cartagena, E-30203 Cartagena, Murcia, Spain. \\
e-mail: pablo.mira@upct.es} \vspace{0.2cm}


\vspace{0.3cm}

\begin{abstract}
We prove that finite area isolated singularities of surfaces with
constant positive curvature $K>0$ in $\R^3$ are removable
singularities, branch points or immersed conical singularities. We
describe the space of immersed conical singularities of such surfaces in
terms of the class of real analytic closed locally convex curves in $\S^2$ with
admissible cusp singularities, characterizing when the singularity is actually embedded.
In the global setting, we describe
the space of \emph{peaked spheres} in $\R^3$, i.e. compact convex
surfaces of constant curvature $K>0$ with a finite number of
singularities, and give applications to harmonic maps and constant
mean curvature surfaces.
\end{abstract}

\section{Introduction}

It is a classical result that any complete surface of constant
curvature $K>0$ in $\R^3$ is a round sphere. Thus, if one extends by
analytic continuation a local piece of such a \emph{$K$-surface} in $\R^3$ other
than a sphere, singularities will eventually appear. It is hence
natural to consider $K$-surfaces in $\R^3$ in the presence of
singularities, and to investigate how the nature of these
singularities determines the global geometry of the surface.

Surfaces of positive constant curvature with singularities are
related to other natural geometric theories. For instance, they are
parallel surfaces to constant mean curvature surfaces, their Gauss
map is harmonic into $\S^2$ for the conformal structure of the
second fundamental form, and this Gauss map can often be realized as
the vertical projection of a minimal surface in the product space
$\S^2\times \R$. Also, when viewed as graphs, these surfaces are
solutions to one of the most widely studied elliptic Monge-Ampère
equations:

\begin{equation}\label{Kequation}
u_{xx} u_{yy} -u_{xy}^2 = K(1+u_x^2 +u_y^2)^2, \hspace{1cm} K>0.
\end{equation}
Thus, the regularity theory for $K$-surfaces in $\R^3$ is tightly
linked to the regularity theory of Monge-Ampère equations.

\begin{figure}
  \begin{center}
   \includegraphics[clip,width=4cm]{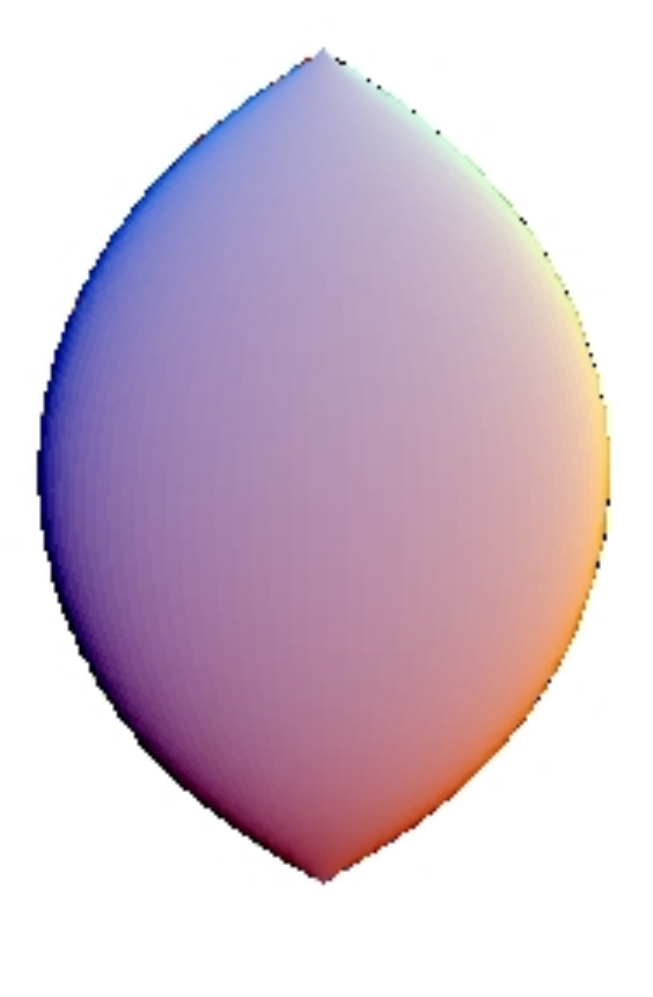}
\caption{A rotational peaked sphere ($K=1$) in $\R^3$}
 \end{center}
\end{figure}

It is remarkable that surfaces of positive constant
curvature in $\R^3$ can be regularly embedded around an isolated singularity,
as shown by the rotational example in Figure 1. This type of singularities does not appear in many other geometric theories, and even for the case of $K$-surfaces in $\R^3$ they only appear in special circumstances. Still, $K$-surfaces in $\R^3$ having only isolated singularities are of central interest to the theory from several points of view. We explain this next, together with our contributions to the topic.

\vspace{0.2cm}

\noindent {\bf 1. Singularity theory.} The singularities of $K$-surfaces in $\R^3$ generically form curves in $\R^3$. Thus, $K$-surfaces only with isolated singularities constitute a rare phenomenon that happens when the surface with singularities has the biggest possible regularity (it is everywhere analytic except for some isolated points).

A natural problem in the theory of singularities is to understand when an isolated singularity is actually \emph{removable} in an adequate sense. In this paper, we will show (Theorem \ref{mainth1}) that an isolated singularity of a $K$-surface in $\R^3$ is \emph{extendable} (i.e. the surface extends smoothly to the puncture with a well-defined tangent plane at it) if and only if the mean curvature is bounded around the singularity. Also, we prove that an isolated singularity of a $K$-surface in $\R^3$ has finite area if and only if it has finite total mean curvature (see Proposition \ref{finito}).

\vspace{0.2cm}

\noindent {\bf 2. Ends of $K$-surfaces in $\R^3$.} Another global restriction of $K$-surfaces in $\R^3$ is that one cannot have complete ends in the theory (in particular, the
exterior Dirichlet problem for \eqref{Kequation} does not have a
solution). This contrasts with the situation in other related
theories, such as flat surfaces in hyperbolic $3$-space \cite{GMMi,GaMi,KUY,KRSUY} or
solutions to the Hessian one equation ${\rm det} (D^2 u)=1$, see
\cite{FMM,GMM,ACG,Mar}.

In any case, Figure 1 also suggests that these isolated singularities can be thought of as the most natural notion of an \emph{end} for a $K$-surface in $\R^3$ (i.e. it is the most regular way in which a non-spherical $K$-surface can end). The study of such an \emph{end}, although it is placed at a finite point in $\R^3$, is non-trivial. For instance, the unit normal does not extend continuously across that point, and the extrinsic conformal structure is that of an annulus (and not of a punctured disk).

In this paper we classify such \emph{ends} under more general assumptions. We will show in Theorem \ref{mainth2} that
any non-extendable isolated singularity of a $K$-surface in $\R^3$ with finite area has a \emph{limit} unit normal at the singularity, which
is a real analytic, closed strictly locally convex curve in $\S^2$, possibly having isolated singular points of cusp type of a certain shape. Conversely, we show that any such curve arises as the limit normal of a unique non-extendable isolated singularity of a $K$-surface in $\R^3$.

In the above classification, we also characterize when the surface is regularly embedded around the singularity. In this way, we prove (Theorem \ref{mainth3}) that the space of embedded isolated singularities of $K$-surfaces in $\R^3$ is in one-to-one
correspondence with the class of real analytic strictly convex
Jordan curves on the sphere $\S^2$, and that $K$-surfaces can be
analytically extended across any embedded isolated singularity by
Schwarzian reflection for the conformal structure of the second
fundamental form.

\vspace{0.2cm}

\noindent {\bf 3. Isometric embedding of abstracts metrics in $\R^3$ and peaked spheres.} A classical problem in differential geometry is to determine when a Riemannian surface $(M^2,g)$ can be isometrically immersed (or embedded) into $\R^3$, and if such an isometric realization is unique up to rigid motions in $\R^3$.  On the other hand, the class of conformal metrics of constant curvature $1$ with conical singularities on a Riemann surface has been widely studied, both from a local and global point of view. It is hence natural to study the isometric embedding problem for such metrics of constant curvature with conical singularities.

In this paper we will show that any non-extendable isolated singularity of finite area of a $K$-surface in $\R^3$ (all of which are classified in Theorem \ref{mainth2}) is, from an intrinsic point of view, a conical singularity, see Proposition \ref{coni}. Hence, our results describe the space of possible local isometric realizations into $\R^3$ of an abstract conical metric of curvature $1$.

In the global case, let us define a \emph{peaked sphere} in $\R^3$ as a closed convex
surface (i.e. the boundary of a convex region of $\R^3$) that is
everywhere regular except for a finite set of points, and has
constant curvature $K>0$ at regular points. These are, from many
different points of view, the most regular $K$-surfaces in $\R^3$
that one can consider globally. Note that a peaked sphere with no
singularities is a round sphere, that there are no peaked spheres
with exactly one singularity, and that a peaked sphere with exactly
two singularities is a rotational example like in Figure 1.

Combining results by Alexandrov, Pogorelov, Troyanov and Luo-Tian,
one can show that peaked spheres with an arbitrary number $n>2$ of
singularities exist, and that they constitute a $3n-6$ dimensional
family. These parameters come from the freedom in prescribing the
intrinsic conformal structure of the peaked sphere (which is that of
a finitely punctured sphere) and the cone angles at the
singularities, all of which are in $(0,2\pi)$ and must obey two
inequalities (see Section \ref{sec:intrinsic}). Thus, peaked spheres
can be well described from an intrinsic point of view, but they are
far from being understood in what regards their actual shape in
$\R^3$. For instance, we do not know necessary and sufficient
conditions under which a finite set of points $p_1, \dots, p_n\in
\R^3$ can be realized as the set of singularities of a peaked sphere
in $\R^3$.

By using the intrinsic classification of peaked spheres (see Theorem \ref{teointri}) and our analysis on embedded isolated singularities of $K$-surfaces in $\R^3$ in terms of their extrinsic conformal structure, we can provide some applications to harmonic maps into $\S^2$ and CMC surfaces in $\R^3$. In particular, we describe the space of solutions to the Neumann problem for harmonic diffeomorphisms into $\S^2$ (Theorem \ref{noiman}), we provide a new \emph{reflection principle} for CMC surfaces in $\R^3$ (Corollary \ref{rprinc}), and we obtain a family of complete branched CMC surfaces in $\R^3$ with an infinite number of ends, obtained by applying this reflection principle to a compact fundamental CMC piece whose boundary meets a finite set of spheres tangentially.

The paper goes as follows. In Section \ref{sec:prelim} we deal with
generalities on surfaces of constant curvature $K>0$ in $\R^3$ with
wave-front singularities in terms of their \emph{extrinsic conformal
structure}, i.e. the conformal structure given by the second
fundamental form of the surface. In Section \ref{sec:local} we prove
the results explained above regarding the classification of isolated singularities of $K$-surfaces in $\R^3$.  In Section \ref{sec:intrinsic} we expose the intrinsic classification
of peaked spheres with $n$ singularities. This result
follows from the classification results by Troyanov and Luo-Tian on
spherical cone metrics on $\S^2$ with cone angles in $(0,2\pi)$,
from Alexandrov's solution to the generalized Weyl's problem, and
from Pogorelov's theorems of local regularity and rigidity for
convex surfaces in $\R^3$. In Section \ref{sec:applic} we transfer
the intrinsic classification of peaked spheres in $\R^3$ to other
geometric theories (harmonic maps into $\S^2$ and CMC surfaces in $\R^3$), by means of their analytic properties with respect to their extrinsic conformal structure and our results of Section \ref{sec:local}. We end up the paper with
some relevant open problems about the extrinsic geometry of peaked
spheres in $\R^3$, and their relation to harmonic maps and constant
mean curvature surfaces.

\section{Surfaces of positive constant curvature in
$\R^3$}\label{sec:prelim}

Let $\psi:M^2\flecha \R^3$ denote an immersed surface with constant
curvature $K>0$ in $\R^3$. Up to a dilation, we shall assume that
$K=1$. Such surfaces will be called from now on \emph{$K$-surfaces} in $\R^3$.

By changing orientation if necessary, the second fundamental
form $II$ of the immersion $\psi$ is positive definite, and thus it
induces a conformal structure on $M^2$. This structure will be
called the \emph{extrinsic conformal structure} of the surface
$\psi$.

Specifically, we may view the surface as an immersion
$X:\Sigma\flecha \R^3$ from a Riemann surface $\Sigma$ such that, if
$z=u+iv$ is a complex coordinate on $\Sigma$ and $N:\Sigma\flecha
\S^2$ denotes the unit normal of the surface, then $\esiz
X_u,N_u\esde = \esiz X_v,N_v\esde <0$ and $\esiz X_u,N_v\esde =0$.
It is easy to check (see also \cite{GaMa}) that the condition $K=1$ implies
 \begin{equation}\label{forbas}
 X_u = N\times N_v, \hspace{1cm} X_v = - N\times N_u.
 \end{equation}
In particular, the unit normal $N:\Sigma\flecha \S^2$ is a harmonic
map into $\S^2$, and the immersion $X$ satisfies the equation

\begin{equation}\label{eqforX}
X_{uu}+X_{vv} = 2 X_u \times X_v.
\end{equation} Observe that $N$ is always a
local diffeomorphism and that, by \eqref{forbas}, the surface $X$ is
uniquely determined by $N$ up to translations.

In order to deal with the possible appearance of singularities for
the surface $X$, it is useful to introduce the following notions
(see \cite{KRSUY,SUY}, for instance).

\begin{defi}
A $C^{\8}$ map $f:M^2\flecha \R^3$ from a surface $M^2$ is called a \emph{frontal} if there exists a $C^{\8}$ map $N:M^2\flecha \S^2$ such that $\esiz df,N\esde =0$. Such a map is called the \emph{unit normal} of the frontal. If in addition $\esiz df,df\esde + \esiz dN,dN\esde$ is positive definite, we say that $f$ is a \emph{front} in $\R^3$.
\end{defi}
Clearly, any regular surface in $\R^3$ is a front, and any front is a frontal. We will use this terminology from now on.

Conversely to the previous discussion on $K$-surfaces in $\R^3$ in terms of their Gauss maps, if $N:\Sigma\flecha \S^2$ is a harmonic map such that
the smooth map $X$ given by \eqref{forbas} is single valued (e.g. if
$\Sigma$ is simply connected), then $X:\Sigma\flecha \R^3$ is a
frontal in $\R^3$, having $K=1$ at its regular
points. Moreover, the singular points of $X$ are precisely the points at which $N$
fails to be a local diffeomorphism. The frontal $X$ will be a front except at the points where $dN$ vanishes.

In terms of a conformal parameter $z=u+iv$ for the second fundamental form, the fundamental forms of
a $K$-surface $X$ are given by
\begin{equation}\label{fundam}
\left\{\def\arraystretch{1.5} \begin{array}{rrc} \esiz dX,dX\esde& =& Q \, dz^2 + 2\mu |dz|^2 + \bar{Q} d\bar{z}^2, \\
-\esiz dX,dN\esde&=& 2\rho |dz|^2,  \\
\esiz dN,dN\esde & = & -Q \, dz^2 + 2\mu |dz|^2 - \bar{Q}
d\bar{z}^2,
\end{array}\right.
\end{equation}
where:
 \begin{enumerate}
 \item
$Q dz^2 := \esiz X_z,X_z\esde dz^2=-\esiz N_z,N_z\esde dz^2$ is a holomorphic quadratic differential on $\Sigma$, which vanishes exactly at the umbilical points of the surface. We will call it the
\emph{extrinsic Hopf differential} of the surface.
 \item
$\mu,\rho :\Sigma \flecha (0,\8)$ are smooth positive real functions, which by the condition $K=1$ satisfy the relation
 \begin{equation}\label{romu}
 \rho^2 = \mu^2 - |Q|^2.
 \end{equation}
 \end{enumerate}
In order to define yet another geometric function of interest on the
surface, let us assume preliminarily that $Q(z_0)\neq 0$ for some
$z_0\in \Sigma$. Then, around $z_0$, the real analytic function
$\omega$ given by

\begin{equation}\label{omro}
\rho = |Q| \sinh \omega
\end{equation}
is well defined, and positive. Moreover, as $Q$ vanishes only at isolated points, $\omega$ can be defined globally on $\Sigma$, if we allow the possibility that $\omega=+\8$ at the umbilical points of the surface. Besides, by the Gauss equation of the surface,
 \begin{equation}\label{lapl}
\omega_{z\bar{z}} + |Q| \sinh \omega=0 \hspace{1cm} \text{(or,
equivalently, }  \omega_{z\bar{z}} + \rho =0).
 \end{equation}
Observe that the mean curvature of $X$ is given by any of these formulas:
 \begin{equation}\label{forH}
 H= \frac{\rho \, \mu}{\mu^2-|Q|^2}=\frac{\mu}{\rho} =\coth (\omega).
 \end{equation}

It is important for our purposes to understand in terms of the Gauss
map how singularities can occur on a surface with $K=1$. In order to
do so, assume first that $N:\Sigma\flecha \S^2$ is a harmonic map
such that $\esiz dN,dN\esde$ is given as in \eqref{fundam}. Let $X:\Sigma\flecha \R^3$ be the $K=1$ frontal in $\R^3$ described by \eqref{forbas}, where we assume that $\Sigma$ is simply connected. It is clear from \eqref{fundam} that $N$
fails to be a local diffeomorphism exactly at the points where $\mu
= |Q|$. Moreover, by \eqref{fundam}, these points agree with
the points at which $X:\Sigma\flecha \R^3$ is not an immersion.

Thus, if $N$ is a local diffeomorphism, $X$ will be an immersion. Let us analyze what happens on a \emph{singular point} of $N$, i.e. a point $z_0\in \Sigma$ such that $dN(z_0)$ has rank $\leq 1$. We will rule out the case that $N$ is singular on an open set of $\Sigma$, since this would imply that $X$ is everywhere singular.

Let $z_0\in \Sigma$ be a singular point of $N$. Then $\mu(z_0)=|Q|(z_0)$. If $\mu (z_0)= |Q|(z_0)\neq 0$, we have a point at which
$dN(z_0)$ has rank one. Consequently, it is well known (see \cite{Woo} for instance) that there exists a regular curve
$\Gamma\subset \Sigma$ passing through $z_0$ such that $dN$ has rank
one at every point of $\Gamma$. It is also direct that the function $\omega$ vanishes identically along $\Gamma$. Contrastingly, if $\mu
(z_0)=Q(z_0)=0$, we have a point where $dN(z_0)=0$. These points are
necessarily isolated, since $Q$ is holomorphic. There are two possible behaviors for the Gauss map around such points (see \cite{Woo}):
 \begin{enumerate}
   \item
$z_0$ is a meeting point of regular curves in $\Sigma$ at which $dN$
has rank one, or
 \item
The Gauss map has a \emph{branch point} at $z_0$. That is, there are local coordinates $(x,y)$ and $(u,v)$ on $\Sigma$ and $\S^2$ centered at $z_0$ and $N(z_0)$, respectively, such that $N$ has the form $u+iv= (x+iy)^k$, for some integer $k>1$.
 \end{enumerate}
In the first case, $z_0$ is not an isolated singularity for the
$K=1$ frontal $X$ given by \eqref{forbas}, since as we explained before, $X$ fails to be an immersion whenever $dN$ has rank $\leq 1$ . In the second case the
$K=1$ frontal has an isolated \emph{branch point}, at which its
first and second fundamental forms are zero. In that case, the
surface $X$ is regularly immersed but not regularly embedded around the branch point. Indeed, if it was embedded, by Lemma \ref{graf} the surface would be a graph around the branch point, and hence the Gauss map would be injective, a contradiction.

Another important fact about surfaces with $K=1$ in $\R^3$ is that
they appear as parallel surfaces of constant mean curvature (CMC)
surfaces in $\R^3$. Specifically, if $X:\Sigma\flecha \R^3$ is a
$K=1$ surface with Gauss map $N:\Sigma\flecha \S^2$, then
$f=X+N:\Sigma \flecha \R^3$ and $f^- =X-N:\Sigma\flecha \R^3$ are
CMC surfaces with $H=-1/2$ and $H=1/2$, respectively, possibly with
singular points. The singular points $z_0\in \Sigma$ for $f$ (resp.
$f^-$) appear when $\mu (z_0)= \rho (z_0)$ (resp. $\mu (z_0)=
-\rho(z_0)$). In both cases, by \eqref{romu}, it holds $Q(z_0)=0$ for
the singular points of $f$ and $f^-$. In fact, as the Hopf
differential vanishes exactly at the umbilical points of the $K=1$
immersion $X$, we see that the umbilical points of $X:\Sigma\flecha
\R^3$ are exactly the points of $\Sigma$ at which one of the two
parallel CMC surfaces has a singular point (the other one will have
an umbilical point).

Let us also remark that if $X:\Sigma\flecha \R^3$ is a frontal of
$K=1$ with singularities, constructed as explained above from a
harmonic map $N:\Sigma\flecha \S^2$, it still makes sense to define
the parallel CMC surfaces, and same criteria for the appearance of
singular points of these CMC surfaces hold.

\subsection*{Rotational peaked spheres in $\R^3$}

Let us consider the rotational $K=1$ surfaces  $X_A (u,v)$ given by $$X_A (u,v)=(f(u) \cos v, f(u) \sin v, h(u)), \hspace{1cm} (u,v)\in \Omega:=(-\pi/2,\pi /2)\times \R,$$ where $f(u)=A\cos u \, $ and $h(u)=\int_0^u \sqrt{1-A^2 \sin^2 u } \, du$, being $A\in (0,1)$ a real parameter (see Figure 1). Clearly, $X_A$ extends continuously to its closure $\bar{\Omega}$, so that $X_A(\pm \pi /2, \R)$ are two points in the axis of rotation.

In this way, we have for each $A\in (0,1)$ a rotational embedded surface with two isolated singularities, that will be called a \emph{rotational peaked sphere}. The distance between these singularities agrees with the extrinsic diameter of the rotational peaked sphere, and varies continuously between $2$ and $\pi$ in terms of $A$. The family $X_A$ varies between a sphere $(A=1)$ and a vertical segment of length $\pi$ ($A=0$). The unit normal $N_A$ of $X_A$ extends continuously to the boundary of $\Omega$, and $\alfa (v)= N_A (\pm \pi /2,v)$ is a circle in $\S^2$ of constant geodesic curvature $k_g = \pm \sqrt{(2-A^2)/(1-A^2)}$.

Consider now the change of coordinates $$s=s(u)= \int_0^u \frac{1}{\sqrt{1-A^2 \sin^2 r}} \, dr, \hspace{1cm} t=v.$$ It follows that $(s,t)$ are global conformal parameters for $X_A$ with respect to the second fundamental form. Using the conformal change $w= e^{z}$ (where $z=s+it$) we see that $X_A$ has the extrinsic conformal structure of an annulus $\mathbb{A}=\{w: |w| \in (r,R)\}$, with $r= e^{-a}$ and $R= e^a$, being $a= \int_0^{\pi/2} 1/\sqrt{1-A^2 \sin^2 u} \, du.$ Thus, the modulus $R/r= e^{2a}$ of this annulus varies with $A$ between $e^{\pi}$ and $\8$. In particular, conformal annuli with conformal modulus in $(1,e^{\pi})$ cannot be realized as the extrinsic conformal structure of a rotational peaked sphere.

The surfaces $X_A$ are parallel to the \emph{unduloids}, i.e. the
rotationally invariant embedded CMC surfaces in $\R^3$ ($H=1/2$ in
our situation).

\subsection*{The geometric Cauchy problem}

Let $\beta(u):I\subset \R\flecha \R^3$ and $V(u):I\subset \R\flecha
\S^2$ denote a real analytic regular curve and a real analytic map
into $\S^2$, with $\esiz \beta',V\esde=0$. By the Cauchy-Kovalevsky theorem, there is a unique solution $X(u,v)$ to the Cauchy problem for the equation \eqref{eqforX} and the initial conditions $X(u,0)=\beta(u)$ and $X_v(u,0)= -V(u)\times V'(u)$. Using this fact and equations \eqref{forbas}, \eqref{fundam} one can easily show that a necessary and
sufficient condition for the existence of a regular $K=1$ surface in
$\R^3$ passing through $\beta$ and whose unit normal along $\beta$
is given by $V$ is that $\esiz \beta',V'\esde \neq 0$ for all $u$.
In these conditions, the solution is necessarily unique. (See \cite{GaMi,GaMi2,ACG} for the solution to this \emph{geometric Cauchy problem} in other geometric contexts).

Moreover, if we remove one of the regularity conditions $\beta'\neq 0$ or $\esiz
\beta',V'\esde \neq 0$, a similar result holds for $K=1$ frontals in
$\R^3$. The resulting surface will have a singular point along
$\beta$ wherever $\esiz \beta'(u_0),V'(u_0)\esde =0$. We can extract
from here two consequences:

\begin{itemize}
  \item
Any regular analytic curve in $\R^3$ with non-vanishing curvature is
realized as a curve of singularities of a unique $K=1$ frontal in
$\R^3$. This is easily seen taking into account that the metric relations $\esiz \beta',V\esde = \esiz \beta',V'\esde=0$ actually determine the field $V$ from the curve $\beta$.
 \item
If $\beta(u)=a\in \R^3$ and $V(u):I\flecha \S^2$ is real analytic, then there exists a unique $K=1$ frontal $X:\Omega\subset
\C\flecha \R^3$, $\Omega$ an open complex domain containing $I$,
such that $X|_I= a\in \R^3$ and its unit normal along the real axis
is given by $V(u)$.
\end{itemize}

This last consequence suggests a method for constructing isolated
singularities of $K=1$ surfaces in $\R^3$ as solutions to singular
geometric Cauchy problems. This idea was first used in the authors'
previous work \cite{GaMi} on flat surfaces in hyperbolic $3$-space,
and it will be fully exploited for the present case in Section 3.

It must be remarked that the unique solution to the geometric Cauchy
problem for $K=1$ surfaces in $\R^3$ explained above can be
described using loop groups. This is a consequence of the
corresponding result for CMC surfaces by Brander and Dorfmeister
\cite{BrDo}, and the relation between CMC surfaces and $K$-surfaces
as parallel surfaces.

\subsection*{A boundary regularity result}

In the next section we will use several times the following boundary
regularity result for solutions to the Dirichlet problem of
\eqref{eqforX}. The result comes from Jacobowsky's paper
\cite[Theorem 4.1]{Jac} (see also the paper by Brezis and Coron
\cite{BrCo}).

We must point out that the boundary regularity result of \cite{Jac}
is only formulated in the continuous case; however, as explained in
the introduction of \cite{Jac}, the result actually provides higher
regularity at the boundary of the solution provided the boundary
condition also has this higher regularity.

\begin{lem}[\cite{Jac}]\label{jacobo}
Let $\Omega\subset \R^2$ be a bounded domain whose boundary $\parc
\Omega$ is $C^{\8}$. Let $X=(X_1,X_2,X_3):\Omega\flecha \R^3$ be a
solution to the Dirichlet problem $$X_{uu} + X_{vv} = 2 \, X_u \times
X_v \hspace{0.6cm} \text{ in $\Omega$}, \hspace{2cm} X=\varphi
\hspace{0.6cm} \text{ on $\parc \Omega$},$$ where $\varphi\in C^{\8}
(\overline{\Omega},\R^3)$. Assume that $X_k\in C^{\8}(\Omega)\cap
C(\overline{\Omega})$ and also that each $X_k$ lies in the Sobolev
space $ H^1 (\Omega) \equiv W^{1,2}(\Omega)$, for $k=1,2,3$. Then
$X_k \in C^{\8}(\overline{\Omega})$.
\end{lem}

\section{The classification of isolated singularities}\label{sec:local}

In what follows, let $D\subset \R^2$ denote a disc of center $q$,
and $D^*:= D\setminus \{q\}$ denote its associated punctured disc.

\begin{defi}
Let $\psi:D^*\flecha \R^3$ denote an immersion of a punctured disc
$D^*$ into $\R^3$, and assume that $\psi$ extends continuously to
$D$. Then, the surface $\psi$ is said to have an \emph{isolated
singularity} at $p=\psi (q) \in \R^3$.

If $\psi$ is an embedding around $q$, $p$ will be called an
\emph{embedded isolated singularity}. The singularity is called
\emph{extendable} if $\psi$ and its unit normal $N$ extend smoothly to $D$, and
\emph{removable} if it is extendable and $\psi:D\flecha \R^3$ is an
immersion.
\end{defi}
An example of a non-extendable embedded isolated singularity is
given by the graph of an arbitrary smooth function on a punctured
planar topological disc $\cU^*$ that extends continuously but not
$C^1$ across the puncture. For example, the two isolated singularities of the surface in Figure 1 are non-extendable embedded isolated singularities.

Our aim is to classify locally the isolated singularities of
$K$-surfaces in $\R^3$. For that, we will identify two
$K$-surfaces having some point $p\in \R^3$ as an isolated
singularity if they overlap on some common neighborhood of $p$. This
is a natural identification since $K$-surfaces in $\R^3$ are
real analytic, and our description in this section will be local
around the isolated singularity $p$.

\begin{pro}\label{finito}
Let $\psi:D^*\flecha \R^3$ denote an immersed $K$-surface in
$\R^3$ with an isolated singularity at $p=\psi(q)$. The following
two conditions are equivalent:

\begin{enumerate}
  \item[(i)]
$\psi$ has finite area around $p$.
 \item[(ii)]
$\psi$ has finite total mean curvature (i.e. $\int |H| dA <\8$) around
$p$.
\end{enumerate}
In that case, $p$ is called a \emph{finite isolated singularity}.
\end{pro}
\begin{proof}
From $H^2-K\geq 0$ and $K=1$, we clearly see that $(ii)$ implies
$(i)$.

Conversely, assume that $\psi$ has finite area around $p$. As we
already said, we can reverse orientation if necessary to assume that
$H\geq 1$. From \eqref{eqforX} and \eqref{fundam} we see that (see
also \cite{GaMa})
$$\lap^{II} \psi= 2N,$$ where $\lap^{II}$ denotes the Laplacian for
the second fundamental form $II$, which is a Riemannian metric.
Therefore, if we consider the function $f=\frac{1}{2}\esiz \psi-p,\psi-p\esde:
D^*\flecha \R$, a simple computation from \eqref{fundam} using that $H=\mu/\rho$ gives
 \begin{equation}\label{jafor1}
\lap^{II} f = 2(H+ \esiz N,\psi-p\esde).
 \end{equation}
Now, by Sard theorem, we know that for almost all values of $r>0$
(sufficiently small), the sphere of $\R^3$ centered at $p$ and of
radius $r>0$ meets $\psi (D^*)$ along a regular, not necessarily
connected, curve. In that way, we may take values $r_1>r_2>0$ for
which the intersection of $\psi(D^*)$ with $\S^2(p,r_1)$ and
$\S^2(p,r_2)$ is regular in the above sense. Also, by making $r_1$
small enough, we may assume that $\psi(\parc D)$ lies outside
$\S^2(p,r_1)$.

Consider next the domain $$\Omega_{r_1,r_2}=\{q\in D^* : r_1\geq
||\psi(q)-p||\geq r_2\}$$ and $C_{r_i}=\{q\in D^* : ||\psi
(q)-p||=r_i\}$. Using then \eqref{jafor1} and the fact that, by
\eqref{fundam}, the area elements $dA$ and $dA_{II}$ for $I$ and
$II$ agree, we have
\begin{equation}\label{jaforunm}
\int_{\Omega_{r_1,r_2}} \lap^{II} f \ dA_{II} = 2
\int_{\Omega_{r_1,r_2}} (H+ \esiz N,\psi-p\esde) \, dA_{II} = 2
\int_{\Omega_{r_1,r_2}} (H+ \esiz N,\psi-p\esde)\,  dA.
\end{equation}
On the other hand, by the divergence theorem,
 \begin{equation}\label{jafor2}
\int_{\Omega_{r_1,r_2}} \lap^{II} f \, dA_{II}= \int_{C_{r_1}}
II(\nabla^{II} f, \nu) ds_{II} + \int_{C_{r_2}} II(\nabla^{II} f,
\nu) ds_{II}.
 \end{equation}
Here $\nabla^{II}$ (resp. $ds_{II}$) is the gradient (resp. the
boundary arc-length) with respect to the Riemannian metric $II$, and
$\nu$ denotes the exterior unit normal for the domain
$\Omega_{r_1,r_2}$, again with respect to $II$. The first term in
the right hand side of \eqref{jafor2} is not interesting for what follows: call it $A_0(r_1)$. The
second term is actually $\int_{C_{r_2}} \nu (f) \, ds_{II}$.

We claim now that $\nu (f) \leq 0$ along $C_{r_2}$. For that, we
only need to observe that, as $\nu$ is the \emph{exterior} unit normal, for each $x\in C_{r_2}$
there is a curve $\alfa(t)$ on $\Omega_{r_1,r_2}\subset D^*$ such
that $\alfa(0)=x$ and $\alfa'(0)= \nu(x)$. Hence, $\nu(f)(x)=
(f\circ \alfa)'(0)$ and as $(f\circ \alfa)(t)\geq r_1 = (f\circ
\alfa)(0)$, we obtain $(f\circ \alfa)'(0)\leq 0$, as wished.

In addition, $$\left| \int_{\Omega_{r_1,r_2}} \esiz N,\psi -p\esde
\, dA \right| \leq \int_{\Omega_{r_1,r_2}} ||\psi - p || \, dA \leq
r_1 \, \cA(\Omega_{r_1,r_2})\leq r_1 \, \cA(D^*),$$ where $\cA$
stands for the area. Putting together this inequality with
\eqref{jaforunm} and \eqref{jafor2} we see that
$$\int_{\Omega_{r_1,r_2}} H \, dA \leq \frac{1}{2} A_0 (r_1)+ r_1 \,
\mathcal{A} (D^*).$$ As the right hand side of this inequality does
not depend on $r_2$, by making $r_2\to 0$ (through an adequate
subsequence, so that the above explained regularity in the
intersection with a ball holds) we obtain that $\int_{D^*} H \, dA $
is finite, as wished.

\end{proof}

It is clear that there exist isolated singularities in $\R^3$ of
infinite area and infinite total mean curvature around the
singularity. However, we do not know if, in the case of $K$-surfaces in $\R^3$, all isolated singularities are actually finite.

\subsection*{Characterization of extendable isolated singularities}

Let us recall that an isolated singularity of a surface $\psi:D^*
\flecha \R^3$ is called \emph{extendable} if both $\psi$ and its
unit normal $N$ extend smoothly across the singularity. This means
that the surface has a well defined tangent plane at the
singularity, although it could be non-regular at it. The following
result characterizes extendable singularities of $K$-surfaces. It
generalizes theorems by Beyerstedt \cite{Bey} and by Heinz and Beyerstedt \cite{HeBe} for the
case of graphs satisfying \eqref{Kequation} on a punctured disk.

\begin{teo}\label{mainth1}
Let $\psi:D^*\flecha \R^3$ be an immersed $K$-surface with an
isolated singularity at $p=\psi(q)$. The following conditions are
equivalent.
 \begin{enumerate}
   \item[(i)]
The isolated singularity $p$ is extendable.
 \item[(ii)]
The mean curvature of $\psi$ is bounded around the singularity.
 \item[(iii)]
$\psi$ has around the singularity the extrinsic conformal structure
of a punctured disk.
\item[(iv)]
The singularity $p$ is removable, or it is a branch point.
 \end{enumerate}
\end{teo}
\begin{proof}
We show first that $(i)$ implies $(iii)$, arguing by contradiction.
For that, assume that the singularity is extendable, but that its
extrinsic conformal structure is that of an annulus. We denote this
conformal parametrization of the surface by $X$, and the conformal
coordinate of the annulus $\A$ by $z=u+iv$.

As $\esiz d\psi,d\psi\esde$ is a smooth quadratic form on $D$, it is
clear that $\psi$ has finite area. So, by Proposition \ref{finito}
and equation \eqref{fundam},
$$\int_{D^*} H \, dA = \int_{\A} \frac{\mu}{\rho} \, \rho \, du dv =
\int_{\A} (\esiz X_u,X_u\esde + \esiz X_v,X_v\esde ) du dv <\8.$$
Now, since $X$ satisfies \eqref{eqforX}, Lemma \ref{jacobo} shows
that $X(u,v)$ can be extended smoothly to the
boundary of $\A$. But by hypothesis, $N$ can also be smoothly extended to the boundary, and it is constant on this
boundary. This implies from \eqref{forbas} and the fact that $X$ is
constant that $dX$ vanishes identically along the boundary of the
annulus. But that would imply from the uniqueness in the solution to
the Cauchy problem for \eqref{eqforX} that $X$ is constant, a
contradiction.

We prove next that $(iii)$ implies $(ii)$. Assume that the extrinsic
conformal structure is that of the punctured disk $\D^*$. We may
assume without loss of generality that the surface is smooth on
$\parc \D\equiv \S^1$. Let $\omega:\D^*\flecha (0,\8]$ be given by
\eqref{omro}. We are going to prove that for every $z\in \D^*$ it
holds $$\omega(z)\geq {\rm min} \{ \omega (\zeta): |\zeta|=1\} >0.$$
This implies from \eqref{forH} that $H$ is bounded on $\D^*$, as
wished.

In order to prove the above inequality, denote $\omega_0 ={\rm min}
\{ \omega (\zeta): |\zeta|=1\}$. Let $\{r_n\}$ be a strictly
decreasing sequence of real numbers $r_n \in (0,1)$, tending to $0$.
Let $h_n$ denote the unique harmonic function on the annulus $$\A_n=
\{z\in \C : r_n \leq |z|\leq 1\},$$ with the Dirichlet conditions
$h_n=\omega_0$ on $\S^1$ and $h_n=0$ on $\{\zeta : |\zeta|=r_n\}$.
By \eqref{lapl}, we see that $\omega_{z\bar{z}}\leq 0$. Thus, by the
maximum principle applied to $\omega$ and $h_n$ we get that
 \begin{equation}\label{4esja}
 0\leq h_n(z) \leq h_{n+1} (z) \leq \omega(z), \hspace{0.7cm} \text{
for every $n\in \N$ and $z\in \C$ with } r_n\leq |z|\leq 1.
 \end{equation}
Thereby, we see that $\{h_n\}$ is an increasing sequence of harmonic
functions, bounded from above by $\omega_0<\8$. So, they converge to
some harmonic function $h$ on $\D^*\cup \S^1$ which is constantly
equal to $\omega_0$ on $\S^1$. But as $h$ is bounded, we deduce that
$h(z) \equiv \omega_0$ on $\overline{\D}$. So, from \eqref{4esja} we
get $\omega(z)\geq \omega_0$ for every $z\in \overline{\D}$, as
desired. This shows that $(iii)$ implies $(ii)$.

In order to prove that $(ii)$ implies $(iv)$, we first remark that
by \cite{HaLa}, if the surface has bounded mean curvature, then it
has finite area. Thus, arguing as in the proof of $(i)\Rightarrow
(iii)$ we have two possibilities:

 \begin{enumerate}
   \item
If the surface has the extrinsic conformal structure of the
punctured disk $\D^*$ , then the finite area condition shows that
$$\int_{\D^*} (\esiz N_u,N_u\esde + \esiz N_v,N_v\esde) \, du dv <\8,$$ i.e.
$N\in \H^1 (\D^*,\S^2)\equiv W^{1,2} (\D^*,\S^2)$. So, by Helein's
regularity theorem \cite{Hel} for harmonic maps into $\S^2$, $N$ can
be harmonically extended to $\D$. The surface $X$ is then extended
accordingly by means of \eqref{forbas}. If $dN$ is non-singular at
$0$, then $X$ is immersed at $0$, and so the singularity is
removable. In contrast, if $dN$ is singular at $0$, a result by Wood
\cite{Woo} gives that, as $dN$ is non-singular in $\D^*$, $N$ has a branch point at $z_0$. Thus, $(iv)$ holds.
 \item
If the surface has the extrinsic conformal structure of an annulus
$\A$, as we explained previously, we may use Lemma \ref{jacobo} to
extend $X(u,v)$ to its inner boundary, so that $X$ is constant
there. But by \eqref{fundam} we get then that $\mu^2 -|Q|^2 =0$ on
this boundary. And as
$$\frac{|Q|^2}{\mu^2-|Q|^2} = H^2 -1 <\8$$ on $\A$, we deduce that $Q$
vanishes on this boundary curve. Thus $Q=0$ everywhere, i.e. the
surface is a piece of a round sphere with the extrinsic conformal
structure of an annulus and that is constant on a boundary curve of
the annulus. This is impossible, and rules out this second case.
 \end{enumerate}
Hence, $(ii)\Rightarrow (iv)$ holds. Finally, that $(iv)$ implies
$(i)$ is immediate.
\end{proof}

\subsection*{The classification of immersed conical singularities in $\R^3$}

We study next the space of non-extendable finite isolated singularities of $K$-surfaces in $\R^3$. Let us point out that, by Theorem \ref{mainth1}, the extrinsic conformal structure around such a singularity is that of an annulus.

We shall use the notation $\A_r= \{z\in \C: 1<|z|<r\}$, where $r>1$. We also need to introduce the following class of curves with singularities in $\S^2$.

\begin{defi}\label{admiscusp}
A smooth map $\alfa:I\subset \R\flecha \S^2\subset \R^3$ is called a \emph{locally convex curve with admissible cusps} if, for every $s\in I$, the quantity $||\alfa'(s)|| k_{\alfa} (s)$ is a non-zero real number. Here $k_{\alfa}(s)$ is the geodesic curvature of $\alfa$ in $\S^2$, i.e. $$k_{\alfa} (s)= \frac{\esiz \alfa''(s),J\alfa'(s)\esde}{||\alfa'(s)||^3}$$ where $J$ denotes the complex structure of $\S^2$.
\end{defi}
It is clear that any regular locally convex curve in $\S^2$ satisfies this property. Indeed, for regular points of
the curve, the condition of the definition is just that $k_{\alfa}\neq 0$, i.e. that $\alfa$ is locally convex
around the point. Let us explain what happens locally around a singular point, i.e. a point $s_0$ with $\alfa'(s_0)=0$.
By definition, we must have
 \begin{equation}\label{limite}
 \lim_{s\to s_0} ||\alfa'(s)|| k_{\alfa} (s)= c_0 \in \R\setminus\{0\},
 \end{equation}
so in particular $k_{\alfa}\to \pm \8$ when $s\to s_0$.

Assume without loss of generality that $\alfa(s_0)\in \S_+^2$, and
let $\pi:\S_+^2\flecha \R^2$ denote the totally geodesic embedding
of $\S_+^2$ into $\R^2$, given by $$\pi (x_1,x_2,x_3)=
\left(\frac{x_1}{x_3},\frac{x_2}{x_3}\right).$$ If we let
 \begin{equation}\label{defbet}
\beta:=\pi \circ \alfa,
 \end{equation}
then $\beta'(s_0)=0$ and $||\beta'(s)|| \kappa_{\beta} (s) \to c\neq
0$ as $s\to s_0$. Here $\kappa_{\beta}$ stands for the curvature of
$\beta$ in the plane. Taking the Taylor series of $\beta$ around
$s_0$ and up to a rotation of $\R^2$ around the origin, it is easy
to see that for the above limit to exist, we must have

\begin{equation}\label{limbe}
 \beta(s)= \left(a (s-s_0)^k, b (s-s_0)^{k+1}\right) + \text{
higher order terms },
\end{equation}
where $a,b\in \R$ are non-zero, and $k\geq 2$. So, equations \eqref{defbet},\eqref{limbe} give the shape that the curve $\alfa$ must have around an admissible cusp point $s_0$.

\begin{teo}\label{mainth2}
Let $\alfa:\S^1\flecha \S^2$ denote a closed, real analytic, locally convex curve with admissible cusps in $\S^2$. Then:
 \begin{enumerate}
 \item[i)]
There exists a unique harmonic map $N:\Omega\flecha \S^2$ from an
open set $\Omega\subset \C$ containing $\S^1$ into $\S^2$, satisfying the initial conditions
 \begin{equation}\label{incon2}
   N|_{\S^1} = \alfa, \hspace{2cm} \left.\frac{\parc N}{\parc {\bf n}}\right|_{\S^1}=0,
 \end{equation}
where $\parc /\parc {\bf n}$ stands for the normal derivative along $\S^1$.
 \item[ii)]
If $X:\Omega\flecha \R^3$ is the map given in terms of $N$ by the
representation formula \eqref{forbas}, then $X$ is single valued, $X(\S^1)=p$ for some $p\in \R^3$, and for $r>1$ sufficiently close to $1$ the restriction of $X$ to $\A_r$ is an immersion.

In particular, $X:\A_r\flecha \R^3$ is an immersed $K$-surface having $p\in \R^3$ as a non-extendable finite isolated singularity.
 \item[iii)]
Conversely, let $\mathcal{S}$ denote an immersed $K$-surface with
a non-extendable finite isolated singularity at $p\in \R^3$. Then, $\cS$ is one of the
surfaces constructed above.
\end{enumerate}
As a consequence, there exists a correspondence between the space of
immersed surfaces with $K=1$ in $\R^3$ having $p\in \R^3$ as a non-extendable finite
isolated singularity and the class of closed, real analytic, locally convex curves with admissible cusps in $\S^2$.
\end{teo}
 \begin{proof}
We shall identify in the usual way $\S^1$ with $\R /(2\pi \Z)$ and $\C\setminus\{0\}$ with $\C/(2\pi \Z)$. It follows then from the
Cauchy-Kowalevsky theorem applied to the equation for harmonic maps
into $\S^2$ that there exists a harmonic map $N(u,v)$ defined on an open set $\cU\subset \R^2\equiv \C$ containing $\R$, such that
 \begin{equation}\label{incon}
N(u,0)=\alfa(u), \hspace{1cm} N_v(u,0)=0,
 \end{equation}
for all $u\in \R$. Moreover, as $\alfa$ is $2\pi$-periodic, so is $N$ (by uniqueness). Assertion $i)$ follows then immediately.

Let us now prove $iii)$. Consider an immersed $K=1$ surface $\mathcal{S}\subset \R^3$
with $p\in \R^3$ a non-extendable finite isolated singularity. Without loss of generality we may assume that
$p=(0,0,0)$.

By Theorem \ref{mainth1}, $\cS$ has the extrinsic conformal structure of an annulus.
This implies that, changing $\mathcal{S}$ by a proper subset of it if necessary, $\mathcal{S}$ can be conformally parametrized with
respect to the second fundamental form by a map $X:\cU^+\subset
\C\flecha \R^3$, so that $\mathcal{S}=X(\cU^+)$, where
$\cU^+:=\{z\in \C: 0 < {\rm Im} z < \delta \}$ for some $\delta>0$, and $X$ is
$2\pi$-periodic. Moreover, $X$ is real analytic and extends
continuously to the boundary $\R\subset
\parc \cU^+$ with $X(u,0)=0$ for all $s\in \R$. Also, observe that
the other boundary component $\R + i\, r$ of $\parc \cU^+$ is not
relevant to our study: we can assume that $X$ extends analytically
across $\R + i \,r$ to a larger open set.

Write now $X(u,v)= (X_1(u,v),X_2(u,v),X_3(u,v))$, where $z=u+iv$.
As the singularity has finite area we have $$\int_{\cU^+} (\esiz
X_u,X_u\esde + \esiz X_v,X_v\esde) du \, dv <\8,$$ and so we see
that $X_1,X_2,X_3$ belong to the Sobolev space $H^1 (\Omega)\equiv
W^{1,2}(\Omega)$.

Now, recall that $X$ is a solution to \eqref{eqforX} that vanishes
on $\R\cap \parc \Omega^+$. In these conditions Lemma \ref{jacobo}
ensures that $X$ extends smoothly up to the boundary. It follows
then that the extension of $X$ to
 \begin{equation}\label{defome}
 \cU :=\{z\in \C : -\delta<{\rm Im} \, z < \delta\}
 \end{equation}
given by $X(u,-v)=-X(u,v)$ is a real analytic map with $X(u,0)=0$. Hence, $$X_u\times X_v :\cU\flecha \R^3$$ satisfies $(X_u\times X_v)(u,-v)=-(X_u\times X_v)(u,v)$ and vanishes along $\R$. By analyticity and a simple power series argument, we can write near $\R$ $$X_u \times X_v (u,v)=\xi (u,v) \, F(u,v),$$ where $\xi:\cU\flecha \R$ is real analytic with $\xi (u,-v)=-\xi (u,v)$ and $\xi (u,0)\equiv 0$, and $F:\cU\flecha \R^3$ is also real analytic with $F(u,v)=F(u,-v)$ and $F\neq (0,0,0)$ along $\R$. Therefore, the unit normal $N$ of $X$ can be analytically extended across $\R$ by $N=F/||F||$. That is, $N$ extends analytically to a map
$N:\cU\flecha \S^2$ as $N(u,v)=N(u,-v)$, and the surface $X$ is recovered in terms of
$N$ by \eqref{forbas}. Thus, denoting $\alfa(u):=N(u,0)$, we see
that $\alfa$ is real analytic, $2\pi$ periodic, and $N$ is the
unique solution to the harmonic map equation into $\S^2$ for the
initial conditions \eqref{incon}.

Let us prove now that $\alpha$ is a locally convex curve with admissible cusps, following Definition \ref{admiscusp}. First, observe that the zeros of $Q$ are isolated on $\cU$, and
that on $\R$ we have $||\alfa'(u)||^2 = |Q(u,0)|$. Also, observe
that the function $\rho$ in \eqref{fundam} vanishes on $\Omega$
exactly at the points in the real axis (since in the general case,
it vanishes at the singular points of $X$). This tells that the function $\omega$ given by \eqref{omro} also satisfies $\omega (u,0)=0$. So, from \eqref{omro}, we have then
 \begin{equation}\label{rouve}
\rho_v(u,0)= ||\alfa'(u)||^2 \omega_v (u,0).
 \end{equation}
In addition, if $J$ denotes the complex structure of $\S^2$, from
\eqref{forbas} we have at $(u,0)$,
$$\def\arraystretch{1.4}\begin{array}{lll} \esiz \alfa'',J\alfa'\esde &=&
\esiz N_{uu},N\times N_u\esde = \esiz N\times N_u,N\times
X_{uv}\esde = \esiz N_u,X_{uv}\esde \\ & = & \displaystyle
\frac{\parc}{\parc v} \left( \esiz N_u,X_u\esde \right) - \esiz
N_{uv},X_u\esde = \rho_v ,\end{array}$$ where we have used that
$X_u(u,0)=0$. If we compare this with \eqref{rouve}, we get the
relation
 \begin{equation}\label{omegav}
   \omega_v(u,0)= ||\alfa'(u)|| k_{\alfa} (u),
 \end{equation}
where here $k_{\alfa} (u)$ stands for the geodesic curvature in
$\S^2$ of $\alfa(u)$. This equation implies that \begin{equation*}
 \lim_{u\to u_0} ||\alfa'(u)|| k_{\alfa} (u)= c_0 \in \R,
 \end{equation*}
for every $u_0\in \R$, and we want to ensure that $c_0\neq 0$. For that, observe that if $c_0=0$ for some
$u_0$, then $\nabla \omega (u_0)=0$. Now, as
$\omega$ satisfies the elliptic PDE \eqref{lapl}, this implies that
there are at least two nodal curves of $\omega$ passing through
$u_0$ (one of which is the real line). But as the zeros of $\omega$
are singular points of the surface, this contradicts the fact that
$X$ is regular on $\cU\setminus \R$, i.e. the fact that the
singularity is isolated.

So, we have proved that $\alfa$ is a locally convex curve with admissible cusps singularities, as desired. This completes the proof of $iii)$.

In order to prove assertion $ii)$, let $X:\cU\flecha \R^3$ be the
map given by the representation formula \eqref{forbas} in terms of
the harmonic map $N:\cU/2\pi \Z\flecha \S^2$ with initial
conditions \eqref{incon}, where $\cU\subset \C$ is given by
\eqref{defome} for some $\delta>0$. It is then immediate that $X(u,0)=p$ for some $p\in
\R^3$, so it follows from \eqref{forbas} and the periodicity of $N$
that $X$ is also $2\pi $-periodic. Since $\alfa(u)$ is locally convex with admissible cusps (see Definition \ref{admiscusp}, the map
$\omega:\Omega\flecha \R$ given by \eqref{omro} satisfies
\eqref{omegav}. Hence, $\omega_v (u,0) \neq 0$ for all $u\in \R$,
and this implies that $\omega \neq 0$ on $\cU\setminus \R$ (by
taking a smaller $\delta>0$ in the definition of $\cU$, if necessary).
As a consequence, if $\cU^+ :=\cU \cap \{ {\rm Im} \, z >0\}$,
then $X:\cU^+ /2\pi \Z\flecha \R^3$ is regular, and therefore it
is a $K=1$ surface in $\R^3$ having $p$ as an isolated singularity. By construction, the singularity is non-extendable (since it has the conformal type of an annulus) and finite (since it has finite area around the singularity). This concludes the proof of $ii)$,
and of the theorem.
 \end{proof}

\begin{remark}
\emph{Theorem \ref{mainth2} shows that any non-extendable isolated
singularity with finite area of a $K$-surface in $\R^3$ has a well
defined limit unit normal at the singularity, which is a real
analytic closed strictly convex curve with admissible cusps in
$\S^2$. In this sense, the following definition is natural.}
\end{remark}

\begin{defi}
A non-extendable finite isolated singularity of a $K$-surface in $\R^3$ will be called an \emph{immersed conical singularity}.
\end{defi}

Let us recall that there exists an \emph{intrinsic} notion of
conical singularity for a Riemannian metric on a punctured disk.
Specifically, a conformal Riemannian metric $\landa |dz|^2$ on $\D^*$ has a
\emph{conical singularity} of angle $2\pi \theta$ at $0$ if $$\landa
= |z|^{2\beta} \, f \, |dz|^2,$$ where $f$ is a continuous positive
function on $\D$ and $\beta=\theta-1 >-1$.

The next result shows that immersed conical singularities of
$K$-surfaces in $\R^3$ are indeed conical from an intrinsic point of
view.

\begin{pro}\label{coni}
Let $\psi:D^*\flecha \R^3$ denote a $K$-surface with an immersed
conical singularity at the puncture $q\in D$. Then its intrinsic
metric has a conical singularity at $q$.
\end{pro}
\begin{proof}
Since the area of the surface around an immersed conical singularity
is finite, in order to show that the intrinsic metric $ds^2$ of
$\psi$ has a conical singularity at $q$ if suffices to show, by
\cite[Proposition 4]{Bry}, that the conformal type of $ds^2$ around $q$ is
that of a punctured disk.

If this is not the case, then we can parameterize conformally (for
the intrinsic metric $ds^2$) a punctured neighborhood $\cU^*$ of $q$
as a quotient $\Omega /\Z$, where $\Omega:= \{z\in \C: 0<{\rm Im
z}<r\}$ for some $r>0$ (here the real axis corresponds to the
singularity). Thus, there is a meromorphic map $g:\Omega\flecha
\bar{\C}$ such that
$$ds^2 = \frac{4 |g'|^2}{(1+|g|^2)^2} \, |dz|^2,$$ and this metric
is well defined on the conformal annulus $\Omega /\Z$.

Define now the curve $$\gamma(t)= i r (1-t):(0,1)\flecha \Omega\subset \C.$$ As by Theorem \ref{mainth3} we can deduce that the metric of the surface extends smoothly (as a tensor, not as a regular metric) to the singularity, i.e. to the real axis, it is clear that the length $L(\gamma)$ of $\gamma$ is finite as $t\to 0$. In particular, any sequence $\{i s_n\}\to 0$, where $0<s_n<r$, has the property that it is a Cauchy sequence for the $K=1$ metric $ds^2$.

So, as $g$ provides a local isometry between $(\Omega, ds^2)$ and $\bar{\C}$ endowed with its canonical spherical metric, there exists $z_0\in \bar{\C}$ such that $\{g(i s_n)\}\to z_0$ for any sequence $\{i s_n\}$ in the above conditions.

Let us show that $g$ can be extended to $(0,1)\subset\R$ so that $g(s)=z_0$ for
all $s\in (0,1)$. This would imply that the meromorphic function $g$ is constant, a
contradiction.

For that, take $r_0\in (0,1)$ and $\{a_n+ib_n\}$ a sequence of points in $\Omega$ that converge to $r_0$ (for the flat metric in $\C$). We also assume that $a_n\in (0,1)$ and that the real sequence $\{b_n\}$ strictly decreases to $0$.

Observe now that the curves $\beta_b$ whose image is $[0,1]\times \{ib\}$ give rise to a
foliation of $\Omega /\Z$ by closed curves, so that the lengths of $\beta_b$ tend to zero when $b\to 0$ (again since $ds^2$ extends smoothly up to the real axis).

Therefore, $$\def\arraystretch{1.3}\begin{array}{lll}
d(g(a_n+ib_n),z_0) &\leq & d(g(a_n+ib_n),g(i b_n)) + d(g(i b_n),z_0) \\
& \leq & d(a_n +i b_n, i b_n) + d(g(i b_n),z_0) \\ & \leq & L(\beta_{b_n}) + d (g(i b_n), z_0),\end{array} $$ where $d$ denotes indistinctly the spherical distance on $\bar{\C}$, or the distance for the $ds^2$ metric on $\Omega /\Z$ (recall that they are isometric via $g$). As the terms on the right part of the inequality tend to zero when $n\to \8$, we deduce that $g(a_n+ i b_n)\to z_0$, as wished.

This is a contradiction, as explained before. So, the conformal
structure must be that of a punctured disk, and $ds^2$ has a conical
singularity.

\end{proof}

\subsection*{The classification of embedded isolated singularities}

We now focus our study on classifying the embedded isolated
singularities of $K$-surfaces in $\R^3$. First, we have:

\begin{lem}\label{graf}
Let $\psi:D^*\flecha \R^3$ be a strictly locally convex surface in
$\R^3$ having $p=\psi (q)$ as an embedded isolated singularity. Then
there is a neighborhood of $q$ such that
$\psi$ is a convex graph (possibly singular at the puncture)
over some plane of $\R^3$.

In particular, every embedded isolated singularity of a $K$-surface
in $\R^3$ has finite area. So, it is a removable singularity or one
of the conical singularities constructed in Theorem \ref{mainth2}.
\end{lem}
\begin{proof}
The first assertion was proved by the first and third author in
\cite{GaMi}. By convexity, it is well known then that $\psi$
has finite area around the singularity. The rest follows from Theorems \ref{mainth1} and
\ref{mainth2}, bearing in mind that a $K$-surface in $\R^3$ cannot
be embedded around a branch point.
\end{proof}

Once here, the next result characterizes which of the curves $\alfa$
in $\S^2$ described by Theorem \ref{mainth2} correspond to embedded
isolated singularities.

\begin{teo}\label{mainth3}
Let $\psi:D^*\flecha \R^3$ denote a $K$-surface with an immersed conical singularity at $p=\psi(q)$. Let $\alfa:\S^1\flecha \S^2$ denote its limit unit normal at the singularity. Then $p$ is an embedded isolated singularity if and only if $\alfa$ is a regular convex Jordan curve in $\S^2$.

In particular, the conical angle at the singularity is given by $2\pi -\cA (\alfa)$, where $\mathcal{A} (\alfa)$ stands
for the area of the smallest region of $\S^2$ enclosed by the convex
Jordan curve $\alfa$.
\end{teo}

\begin{proof}
Let $X:\cU\flecha \R^3$ be a $K=1$ surface with a conical
singularity, as constructed in Theorem \ref{mainth2}, and assume
that it is embedded. By Lemma \ref{graf}, we can assume that
$X(\cU^+)$ is a convex graph in the $x_3$-axis direction. Hence,
making $\cU^+$ smaller if necessary, it is clear that there exists a
compact convex body $K\subset \R^3$ such that $X(\cU^+)$ is a piece
of its boundary.

Besides, it is well known that the set of interior unit support
vectors at an arbitrary boundary point $p_0$ of a convex body $K$ is
a convex set of the unit sphere $\S^2$. In particular, if the
boundary of $K$ is $C^1$ at $p_0$, this convex set is just a point
in $\S^2$, which agrees with the Gauss map of $K$ at $p_0$.

In this way, as $\alfa(\S^1)$ is the limit set of unit normals of
$X(\cU^+)$, we easily see that $\alfa(\S^1)\subset \S^2$ is the
boundary of a convex set of $\S^2$.  This ensures that the local
behavior \eqref{limbe} is impossible for the curve $\beta=\pi \circ
\alfa$. So, from Theorem \ref{mainth2} we see that $\beta$ (and thus
$\alfa$) is regular and strictly locally convex, and a simple
topological argument ensures that it is actually a convex Jordan
curve (since it is a limit of locally convex Jordan curves in
$\S^2$).

Therefore, we conclude that the curve $\alfa$ is a regular convex
Jordan curve in $\S^2$, as wished.

For the converse, we will use the \emph{Legendre transform} (see
\cite{LSZ})

\begin{equation}\label{leg}
 \cL_X =
\left(\frac{-N_1}{N_3},\frac{-N_2}{N_3}, -X_1 \frac{N_1}{N_3} - X_2
\frac{N_2}{N_3} -X_3\right):\Omega^+\flecha \R^3,
 \end{equation}
where $X=(X_1,X_2,X_3)$ and $N=(N_1,N_2,N_3)$. It is classically
known that $\cL_X$ can be defined for convex multigraphs in the
$x_3$-axis direction, so that $\cL_X$ is also a convex multigraph in
the $x_3$-axis direction. The interior unit normal of $\cL_X$ is
 \begin{equation}\label{uleg}
 \cN_{\cL} = \frac{1}{\sqrt{1+X_1^2 +X_2^2}} \left( -X_1,-X_2,1\right).
 \end{equation}

So, assume that the $K$-surface $X:\cU\flecha \R^3$ with a conical
singularity is generated following Theorem \ref{mainth2} from a
regular convex Jordan curve $\alfa$ in $\S^2$. We prove below that
$X(\cU^+)$ is a graph, taking $\delta>0$ in \eqref{defome} smaller
if necessary, what concludes the proof.

Since $\alfa$ is a convex Jordan curve,
it lies on a hemisphere of $\S^2$, say, $\S_+^2 =\{(x_1,x_2,x_3)\in
\S^2: x_3 > 0\}$, and so $X|_{\cU^+}$ is a local graph in the
$x_3$-axis direction.

Let now $\cL_X:\cU^+ /2\pi \Z\flecha \R^3$ denote the Legendre
transform \eqref{leg} of $X$. It turns out that $\cL_X (\R /2\pi
\Z)$ is a regular convex Jordan curve in the $x_1,x_2$-plane, and
that the unit normal of $\cL_X$ along $\R$ is $(0,0,1)$, constant.

Therefore, $\cL_X$ lies in the upper half-space $\R_+^3$, and there
is some $\ep_0>0$ such that for every $\ep \in (0,\ep_0)$ the
intersection $\Upsilon_{\ep}= \cL_X (\cU /2\pi \Z )\cap \{x_3 =\ep\}
$ is a regular convex Jordan curve. Consider now $S_{\ep_1,\ep_2}$
the portion of $\cL_X$ that lies in the slab between the planes
$\{x_3=\ep_1\}$ and $\{x_3=\ep_2\}$, where $0<\ep_2<\ep_1<\ep_0$.
Then, as $S_{\ep_1,\ep_2}$ is convex and the curves $\Upsilon_{\ep}$
are convex Jordan curves, we get that the unit normal $\cN_{\cL}$ of
$\cL_X$ in this slab is a global diffeomorphism onto its image in
$\S^2$. Letting $\ep_1 \to 0$ and choosing $\delta>0$ sufficiently
small, we get that  $\cN_{\cL}$ is a global diffeomorphism from
$\cU^+ /2\pi \Z$ onto its spherical image in $\S^2$.

Consequently, by \eqref{uleg}, $X(\cU^+)$ is a graph over a
region in the $x_1,x_2$-plane. This concludes the proof of the first statement of the theorem.

In order to compute from an extrinsic point of view the angle of
this conical singularity, we consider the limit tangent cone $C:= t
\, \nu(s)$ at the singularity, where $t>0$ and $\nu(s)$ is a convex
curve on $\S^2$. The unit normal of this tangent cone at the origin
agrees with the limit unit normal $\alfa(s):\R /(2\pi \Z)\flecha
\S^2$ of the surface at the singularity, which is also a convex
Jordan curve. Let us recall that, by Theorem \ref{mainth2}, the
surface is uniquely determined by $\alfa(s)$, which can also be
chosen arbitrarily. Parameterizing $\alfa$ by arc-length, we get
working on $C$, that
$$\nu(s)= \alfa(s)\times \alfa'(s).$$ The cone angle $\theta$ is given by the
length of the convex curve $\nu(s)$, and this implies by the
Gauss-Bonnet theorem that $$\theta = \int ||\nu' (s)|| ds = \int k_g
(\alfa(s)) ds = 2\pi - \mathcal{A} (\alfa)\in (0,2\pi),$$ where $k_g
(\alfa(s))$ and $\mathcal{A} (\alfa)$ denote, respectively, the
geodesic curvature of $\alfa(s)$ in $\S^2$ and the area of the
smallest spherical region enclosed by $\alfa$.
 \end{proof}

Putting together Lemma \ref{graf} and Theorem \ref{mainth3}, we
conclude:
\begin{cor}
The space of non-removable embedded isolated singularities of $K=1$
surfaces in $\R^3$ is in one-to-one correspondence with the class of
real analytic regular convex Jordan curves in $\S^2$.

This correspondence assigns to each embedded isolated singularity its associated limit unit normal at the singularity.
\end{cor}

\section{The intrinsic classification of peaked
spheres}\label{sec:intrinsic}

\begin{defi}
A \emph{peaked sphere} in $\R^3$ is a closed convex surface
$S\subset \R^3$ (i.e. the boundary of a bounded convex set of
$\R^3$) that is a regular surface everywhere except for a finite set
of points $p_1,\dots, p_n\in S$, and such that $S\setminus
\{p_1,\dots, p_n\}$ has constant curvature $1$.

The points $p_1,\dots, p_n$ are called the \emph{singularities} of
the peaked sphere $S$.
\end{defi}
Equivalently, a peaked sphere can also be defined as an embedding
$$\phi:\S^2\setminus \{q_1,\dots, q_n\}\flecha \R^3$$ of constant
curvature $1$, such that $\phi$ extends continuously to $\S^2$. If
$\phi$ does not $C^1$-extend across $q_j$, then $p_j :=\phi (q_j)\in
\R^3$ is a singularity of $S:= \phi (\S^2)\subset \R^3$. That these
two definitions agree follows from a simple topological argument and
the local convexity of $K=1$ surfaces in $\R^3$.

It is clear from our analysis in Section \ref{sec:local} that the
singularities of a peaked sphere in $\R^3$ are conical, with conic
angles in $(0,2\pi)$. Thus, from an intrinsic point of view, peaked
spheres in $\R^3$ are well known objects.

There are no peaked spheres with exactly one singularity. For the
case of two singularities, there exist rotational peaked spheres,
and a simple application of Alexandrov reflection principle shows
that any peaked sphere with exactly two singularities is one of
these rotational examples. So, there is exactly a $1$-parameter
family of non-congruent peaked spheres with $n=2$ singularities, all
of them rotational (see Section \ref{sec:prelim} and Figure 1).

For $n>2$, peaked spheres in $\R^3$ with $n$ singularities exist,
and can be classified from an intrinsic point of view. This follows
from some classical results by Alexandrov and Pogorelov (see also
\cite{BuSh}) on the isometric realization and regularity in $\R^3$
of singular metrics of non-positive curvature, together with the
intrinsic classification of cone metrics of constant positive
curvature on $\S^2$ whose cone angles lie in $(0,2\pi)$ by Troyanov
\cite{Tro} and Luo-Tian \cite{LuTi}. Specifically, the
classification is:

\begin{teo}\label{teointri}
Let $\Lambda$ denote a conformal structure of $\S^2$ minus $n$
points, $n>2$, and let $\theta_1,\dots, \theta_n\in (0,1)$. Then, a
necessary and sufficient condition for the existence of a peaked
sphere $S\subset \R^3$ with $n$ singularities $p_1,\dots,p_n$ of
given conic angles $2\pi \theta_1,\dots, 2\pi \theta_n$, and such
that $\Lambda$ is the conformal structure of $S\setminus
\{p_1,\dots, p_n\}$ for its intrinsic metric, is that
 \begin{equation}\label{angles}
 n-2<\sum_{j=1}^n \theta_j <n-2 +{\rm min}_j \{\theta_j\}.
 \end{equation}
Moreover, any peaked sphere in $\R^3$ is uniquely determined up to
rigid motions by the conformal structure of $S\setminus \{p_1,\dots,
p_n\}$ and by the cone angles $2\pi \theta_1,\dots, 2\pi \theta_n$.

In particular, the space of peaked spheres in $\R^3$ with $n>2$
singularities is a $3n-6$ parameter family, modulo rigid motions.
\end{teo}

\begin{proof}
In \cite{Tro}, Troyanov proved that \eqref{angles} is a sufficient
condition for the existence of a metric of constant curvature $1$ on
$\S^2\setminus \{q_1,\dots, q_n\}$, and such that this metric has at
each $q_j$ a conical singularity of angle $2\pi \theta_j$. Then, Luo
and Tian proved in \cite{LuTi} that if $\theta_j \in (0,1)$ for all
$j$, then \eqref{angles} is also a necessary condition for the
existence of such a metric. Moreover, this metric is unique under
the above hypotheses.

So, in order to prove Theorem \ref{teointri} it suffices to show
that all these metric are isometrically embeddable into $\R^3$ as
peaked spheres, and that the intrinsic metric of any peaked sphere
is isometric to one of the Troyanov-Luo-Tian cone metrics on the
sphere, with all angles in $(0,2\pi)$. The second property is clear,
since we showed in Section \ref{sec:local} that any non-removable
embedded isolated singularity of a $K=1$ surface in $\R^3$ is a
conical singularity of angle in $(0,2\pi)$.

As regards the isometric realization in $\R^3$ of these abstract
cone metrics, Alexandrov proved in \cite{Ale} the following solution
to the generalized Weyl's embedding problem: \emph{any $2$-manifold
with a singular metric of non-negative curvature homeomorphic to a
sphere is isometric to a closed convex surface in $\R^3$;
conversely, any closed convex surface in $\R^3$ is, intrinsically, a
$2$-manifold of non-negative curvature}.

A precise definition of the concept of a $2$-dimensional manifold of
non-negative curvature in the Alexandrov sense, together with an
explanation of this deep theorem, can be consulted in page 24 of
\cite{BuSh}.

It turns out that $\S^2$ endowed with any of the cone metrics on
$\S^2$ with angles in $(0,2\pi)$, classified by Troyanov and
Luo-Tian, is a manifold of non-negative curvature in the Alexandrov
sense. This is just a consequence of the convexity condition
$\theta_j \in (0,1)$ at the singularities. An alternative proof can
be given as follows. First, it follows from the work of Bryant
\cite[Proposition 4]{Bry} that any $2$-dimensional Riemannian metric
of constant positive curvature around a conical singularity is
isometric to a radial Riemannian metric. Moreover, if the conic
angle belongs to $(0,2\pi)$, such a radial Riemannian metric can be
realized in $\R^3$ as the first fundamental form of some
rotationally invariant $K=1$ surface. But this implies by convexity
and embeddedness of this rotational surface that the original
abstract Riemannian metric has non-negative curvature in the
Alexandrov sense, see Theorem 2.2.1 in page 24 of \cite{BuSh}.

As a consequence, if $g$ is a cone metric of constant curvature $1$
on $\S^2$ whose cone angles are all in $(0,2\pi)$, then there exists
a (singular) closed convex surface $S\subset \R^3$ that is isometric
to $(\S^2,g)$. We claim that $S$ is a peaked sphere. Indeed, this
follows from Pogorelov's regularity theorem in \cite{Pog2} (see
Theorem 3.1.1. in page 27 of \cite{BuSh}), as we explain next. As
the metric $g$ is everywhere regular, except for a finite number of
points $q_1,\dots, q_n\in \S^2$, the regularity theorem of Pogorelov
ensures that $S$ is a regular, smooth surface everywhere except on a
finite set of points $p_1,\dots, p_n\in S$. Obviously, in the
regular part of $S$, the Gaussian curvature is $K=1$. Thus, $S$ is a
peaked sphere, as claimed. This completes the proof.
\end{proof}

\section{Peaked spheres and harmonic maps}\label{sec:applic}

\subsection*{The Neumann problem for harmonic diffeomorphisms into
$\S^2$}

Let $\Omega\subset\R^2$ be a bounded domain whose boundary $\parc \Omega$
consists of a finite number of regular Jordan curves. The Neumann
problem for harmonic diffeomorphisms asks about the existence (and
uniqueness) of a harmonic map $g:\Omega\cup
\parc \Omega\flecha \S^2$ that is a diffeomorphism onto its image, and such that
 \begin{equation}\label{noicon}
   \left. \frac{\parc g}{\parc
\mathbf{n}} \right|_{\parc \Omega}  =0 \hspace{1cm}
\text{($\mathbf{n}$ is the exterior normal derivative along $\parc
\Omega)$.}
 \end{equation} We remark that this problem is conformally invariant,
i.e. only the conformal equivalence class of the complex domain
$\Omega$ matters for the problem. So, this domain can always be
assumed to be a \emph{bounded circular domain}, i.e. a disk
$D_1\subset \C$ with a finite collection of interior disjoint disks
removed.

\begin{teo}\label{noiman}
A harmonic map $g:\Omega\flecha \S^2$ is a solution to the Neumann
problem for harmonic diffeomorphisms if and only if it is the Gauss
map of a peaked sphere in $\R^3$, with respect to its extrinsic
conformal structure.

As a consequence, the spaces of harmonic maps into $\S^2$ that solve
the above Neumann problem for some planar domain with $n>2$ boundary
components is a $3n-6$ dimensional family (here the planar domain
$\Omega$ is \emph{not} fixed; only the number $n$ is).
\end{teo}
\begin{proof}
Let $S\subset \R^3$ denote a peaked sphere in $\R^3$. As explained
in Theorems \ref{mainth2} and \ref{mainth3}, $S$ has around any of
its singularities the conformal structure of an annulus with respect
to the second fundamental form. By uniformization, the conformal
type of $S\subset \R^3$ for the second fundamental form is that of a
bounded circular domain $\Omega$. Also, the Gauss map
$N:\Omega\flecha \S^2$ is a harmonic map, which is a diffeomorphism
onto its image. Moreover, when parameterized conformally for its
extrinsic conformal structure, the peaked sphere $X:\Omega\flecha
\R^3$ satisfies that $X$ is constant along each boundary component.
So, by \eqref{forbas}, its Gauss map satisfies the Neumann condition
\eqref{noicon}.

Conversely, let $N:\Omega\flecha \S^2$ denote a solution to the
Neumann problem for harmonic diffeomorphisms into $\S^2$, where
$\Omega$ is a bounded circular domain. Let $X:\Omega\flecha \R^3$
denote the surface with $K=1$, possibly with singularities,
determined from the Gauss map by the representation formula
\eqref{forbas}. Clearly, $X$ is regular and constant along each
boundary circle $C_j$ of $\Omega$, so we only need to ensure that
$X$ is single valued. In order to do this, as the fundamental group
of $\Omega$ is generated by the boundary circles $C_j$, it suffices
to show that $X$ is single valued around each of these circles. But
this property was already proved in the local classification theorem
of conical singularities, see the existence part in Theorem
\ref{mainth2}. So, this concludes the proof.
\end{proof}

As a consequence of this and the fact that peaked spheres in $\R^3$
with two singularities are rotational and their conformal structure
is controlled (see Section 2), we have

\begin{cor}
Let $\mathbb{A}(r,R)$ be the annulus $\{z: r<|z|<R\}$. Then, the
Neumann problem for harmonic diffeomorphisms
$g:\mathbb{A}(r,R)\flecha \S^2$ has a solution if and only if
$R/r>e^{\pi}$.

In that case, the solution is unique and radially symmetric.
\end{cor}

\subsection*{CMC surfaces with free boundary}

A well studied problem in the theory of CMC surfaces is the free
boundary problem (or \emph{capillarity} problem) of finding all
compact CMC surfaces that meet a certain support surface $S\subset
\R^3$ at a constant angle along their boundary. Our analysis on
$K=1$ surfaces in $\R^3$ with isolated singularities provides some
interesting consequences in this context.

First, one has the following reflection principle in the spirit of
the usual Schwarz's reflection principle (i.e. that CMC surfaces
meeting a plane orthogonally along its boundary can be analytically
extended by reflection across this plane).

Let $\cU$ be a bounded symmetric domain in $\C$, i.e.
$\cU=\cU^*:=\{\bar{z}: z\in \cU\}$. Assume that $\cU\cap \R\neq
\emptyset$ and call $\cU^+ :=\cU\cap \C_+$.

\begin{cor}\label{rprinc}
Let $f:\cU^+\flecha \R^3$ be a conformally immersed $H=1/2$ surface
in $\R^3$. Assume that $f$ extends $C^1$ to
$\Gamma:=\overline{\cU^+}\cap \R$, so that $f|_{\Gamma}$ is
contained in the sphere $\S^2(1)$ of radius one, and meets this
sphere tangentially.

Then, $f$ can be analytically extended to $\cU$ by the formula
\begin{equation}\label{reprin}
f(\bar{z})=-f(z)-2 N(z),
\end{equation}
where here $N:\cU^+\flecha \S^2$ is the unit normal of $f$.
\end{cor}
\begin{proof}
Consider the parallel $K=1$ frontal $X=f+N:\cU^+\flecha \R^3$. Then
$X$ extends continuously to $\Gamma$, with
$X|_{\Gamma} =0\in \R^3$. Observe that from the condition $K=1$, the
area of $X$ agrees with the $\S^2$-area of the spherical image
$N(\overline{\cU^+})$, which is finite around any point $(u_0,0)\in
\Gamma$ since $N$ is continuous. Thus, we are in the conditions of
Lemma \ref{jacobo} around any such point, and so $X$ can be extended
across $\Gamma$ by $X(u,-v)=-X(u,v)$ (see the proof of Theorem
\ref{mainth2}). The Gauss map of $X$ (and $f$) is extended by
$N(u,-v)=N(u,v)$. Once here, formula \eqref{reprin} follows
directly.
\end{proof}

Interestingly, the reflected part described by \eqref{reprin} is the
parallel $H=-1/2$ surface of the original $H=1/2$ surface
$f:\cU^+\flecha \R^3$, composed with the central isometry of $\R^3$
about the center of the sphere. Indeed, if $f^{\sharp}$ is the
parallel $H=-1/2$ surface of $f$, we see from \eqref{reprin} that
$f(\bar{z})= - f^{\sharp} (z)$.

So, the extended surface $f:\cU\flecha \R^3$ is \emph{self-parallel}
(a CMC surface $f:\Sigma\flecha \R^3$ is \emph{self-parallel} if
there exists an antiholomorphic diffeomorphism $J:\Sigma\flecha
\Sigma$ and an orientation preserving rigid motion $\Phi$ of $\R^3$
such that $\Phi \circ f^{\sharp} = f\circ J$, where
$f^{\sharp}:\Sigma\flecha \R^3$ is the parallel CMC surface of $f$).
Thus, we obtain a very general procedure for constructing
self-parallel CMC surfaces in $\R^3$, just by solving an adequate
geometric Cauchy problem.

In addition, the $H=1/2$ surfaces that are parallel to peaked
spheres with $K=1$ in $\R^3$ also have interesting global
properties. Indeed, by the above discussion, they are locally convex
$H=1/2$ immersions $$f:\Omega\subset \C\flecha \R^3,$$ where
$\Omega\subset \C$ is a bounded circular domain with $n\geq 2$
boundary components, satisfying:

\begin{enumerate}
  \item
Each boundary curve $\Gamma_j:=f(C_j)$ (where $\parc \Omega=C_1\cup
\cdots \cup C_n$) is a regular convex Jordan curve contained in a
sphere $\S^2(p_j;1)$ of radius one centered at some $p_j\in \R^3$,
$j=1,\dots, n$.
 \item
The surface meets $\S^2(p_j;1)$ tangentially along $\Gamma_j$ for each
$j=1,\dots, n$.
\end{enumerate}
So, these CMC surfaces are solutions to a free boundary (or
capillarity) problem where the support surface $S$ is a collection
of $n\geq 2$ spheres of radius one in $\R^3$, and the surface meets
this configuration tangentially along all its boundary components.

When $n=2$, the peaked sphere is rotational, and so the parallel CMC
surface $S_H$ is a compact convex piece of an unduloid, bounded by
two parallel circles (actually, the largest piece of the unduloid
with these conditions).

When $n>2$, by the results of Section 4, the space of such $H=1/2$
surfaces $S_H$ with spherical boundaries is a $3n-6$ dimensional
family.

To all these $H=1/2$ surfaces $S_H$ that are parallel to peaked
spheres in $\R^3$, the reflection principle of Corollary
\ref{rprinc} applies. When $n=2$ we obtain the complete unduloid, by
repeatedly applying this reflection principle. In contrast, when
$n>2$ we obtain complete branched CMC surfaces with an infinite
number of ends, again by reflection across all boundary components
and iterating this process.

It must be emphasized that, although the starting compact surface
with boundary $S_H$ is regular (i.e. free of branch points), its
reflected part will always encounter branch points. This is a
consequence of the following facts.
 \begin{enumerate}
   \item
The reflected piece $S_H^*$ of $S_H$ across a boundary curve in the
sense of Corollary \ref{rprinc} is its parallel surface, up to a
rigid motion. So $S_H^*$ will be an immersion (i.e. free of branch
points) if and only if $S_H$ is free of umbilic points.
 \item
Consider the \emph{double Riemann surface} $\bar{M}_g$ of $\Omega$,
which is a compact Riemann surface of genus $g=n-1$. Let $Q \, dz^2$
denote the Hopf differential of $S_H$, which is defined on $\Omega$.
As ${\rm Im} Q \, dz^2$ vanishes along $\parc \Omega$, we can extend
it to be a holomorphic quadratic differential on the compact Riemann
surface $\bar{M}_g$, by Schwarz reflection principle. But as $n>2$,
the genus of the surface is greater than one. This means that $Q\,
dz^2$ must vanish somewhere on $\bar{M}_g$, and thus by symmetry we
conclude that $Q$ vanishes somewhere in $\Omega$, i.e. that $S_H$ is
not free of umbilic points.
 \end{enumerate}
As a consequence, we get the existence of a $3n-6$ dimensional
family (for $n>2$) of complete, branched self-parallel CMC surfaces in $\R^3$ with genus
zero and an infinite number of ends.

\subsection*{Open problems}

Let us recall that the extrinsic conformal structure (i.e. the
conformal structure induced by the second fundamental form) of a
peaked sphere is that of a bounded circular domain in $\C$.

\vspace{0.3cm}

\noindent {\bf Problem 1.} Which bounded circular domains in $\C$
are realizable as the extrinsic conformal structure of a peaked
sphere in $\R^3$? Is a peaked sphere uniquely determined by its
extrinsic conformal structure?

\vspace{0.3cm}

Problem 1 has a strong connection with the Neumann problem for
harmonic diffeomorphisms. Indeed, a classification of peaked spheres
in terms of their extrinsic conformal structure would solve
completely the Neumann problem for harmonic diffeomorphisms into
$\S^2$. Problem 1 follows the spirit of some previous classification
theorems of entire solutions to elliptic PDEs with a finite number
of singularities, in terms of some underlying conformal structure.
See \cite{GMM,FLS,Fer,CMM}.

\vspace{0.3cm}

\noindent {\bf Problem 2.} Find necessary and sufficient conditions
for a set of points $p_1,\dots, p_n\in \R^3$ to be realized as the
set of singularities of a peaked sphere in $\R^3$. Are two peaked
spheres with the same singularities $p_j\in \R^3$  necessarily the
same?

\vspace{0.3cm}

Problem 2 is connected with the free boundary capillarity problem
for CMC surfaces, in the case that one wishes to prescribe the
centers of the spheres (and not just the number of spheres and their
common radius). In any case, it is clear that an arbitrary
configuration of $n$ points will not be in general the singular set
of a peaked sphere in $\R^3$.

\vspace{0.3cm}

\noindent {\bf Problem 3.} Can one realize any conformal metric of
constant curvature $1$ on $\S^2$ with a finite number of conical
singularities as the intrinsic metric of an immersed $K=1$ surface
in $\R^3$?

\vspace{0.3cm}

The results by Alexandrov and Pogorelov show that these metrics are
realized as the intrinsic metric of peaked spheres, \emph{provided}
all conical angles are in $(0,2\pi)$, but there are many other
abstract cone metrics. Such an isometric realization in $\R^3$ must
necessarily be non-embedded. It must be emphasized that, by our
local study, any conical singularity of arbitrary angle can be
realized as an immersed $K=1$ surface in $\R^3$ in many different
ways.

Let us also point out that a complete classification for conformal
metrics of positive constant curvature on $\S^2$ with $n$ conical
singularities remains open if $n>3$ (see \cite{UmYa,Ere} for the
case of three conical singularities).

\def\refname{References}

\end{document}